\documentclass[11pt]{amsart}

\usepackage{geometry}
\usepackage{palatino}
\usepackage[all]{xy}
\usepackage{amsthm,amssymb,amsmath}
\usepackage{mathrsfs,amscd}
\usepackage{multirow}
\newtheorem{theorem}{Theorem}[section]
\newtheorem{conjecture}[theorem]{Conjecture}

\newtheorem{corollary}[theorem]{Corollary}
\newtheorem{lemma}[theorem]{Lemma}
\newtheorem{proposition}[theorem]{Proposition}
\newtheorem{example}[theorem]{Example}
\newtheorem{remark}[theorem]{Remark}
\newtheorem{algorithm}[theorem]{Algorithm}

\newcommand{\bb}[1]{\mathbb{#1}}
\newcommand{\mc}[1]{\mathcal{#1}}
\newcommand{\mf}[1]{\mathfrak{#1}}
\begin{document}

\title[Symplectic Nilpotent Orbits]{Quantization of Special Symplectic Nilpotent Orbits and Normality of their Closures}
\author{Kayue Daniel Wong}

\begin{abstract}
We study the regular function ring $R(\mc{O})$ for all symplectic nilpotent orbits $\mc{O}$ with even column sizes. We begin by recalling the quantization model for all such orbits by Barbasch using unipotent representations. With this model, we express the multiplicities of fundamental representations appearing in $R(\mc{O})$ by a parabolically induced module. Finally, we will use this formula to give a criterion on the normality of the Zariski closure $\overline{\mc{O}}$ of $\mc{O}$.
\end{abstract}
\maketitle
\section{Introduction}
Let $G = Sp(2n,\bb{C})$ be the complex symplectic group. The $G$-conjugates of a nilpotent element $X \in \mf{g}$ form a \textbf{nilpotent orbit} $\mc{O} \subset \mf{g}$. The idea of the orbit method, first proposed by Kirillov, suggests that one can `attach' a unitary representation to each nilpotent orbit $\mc{O}$. In \cite{W 2016}, the author used tools from unipotent representations and dual pair correspondence to achieve this goal for spherical, special nilpotent orbits and their covers (see Theorem A of \cite{W 2016}). In the following work, we would like to study a larger class of nilpotent orbits.\\

It is well-known that all such nilpotent orbits are parameterized by partitions, where the partition corresponds to the size of the Jordan blocks. For $Sp(2n,\mathbb{C})$, nilpotent orbits are identified with the partitions of $2n$ in which odd parts occur with even multiplicity.\\

In fact it is sometimes more convenient to look at the column sizes, or the \textbf{dual partition}, of a given partition. More precisely, let $\psi = [r_1,r_2,\dots,r_i]$ be a partition of $n$, with $r_1 \geq r_2 \geq \dots \geq r_i > 0$, its dual partition is given by $\psi^* = (c_k,c_{k-1},\dots,c_1)$, where $c_{k+1 - j} = \#\{i | r_i \geq j \}$. We will use square bracket $[r_1, r_2, \dots]$ to denote the partition of a nilpotent orbit, and round bracket $(c_k, c_{k-1}, \dots)$ to denote the dual partition of the same orbit.\\

Given two partitions $\varsigma = [r_1,\dots,r_p]$, $\psi = [r_1',\dots,r_q']$, we define their \textbf{union} $\varsigma \cup \psi = [s_1,\dots,s_{p+q}]$, where $\{s_1 \dots, s_{p+q} \} = \{ r_1, \dots, r_p, r_1', \dots, r_q'\}$ as sets and $s_1 \geq s_2 \geq \dots \geq s_{p+q}$. Also, define the \textbf{join} to be $\varsigma \vee \psi = (\varsigma^* \cup \psi^*)^*$, so that if $\varsigma = (c_m,\dots,c_0)$, $\psi = (c_n',\dots,c_0')$, then $\varsigma \vee \psi = (d_{m+n+1},\dots,d_0)$, where $\{d_{m+n+1} \dots, d_0 \}$ $=$ $\{ c_m, \dots, c_0, c_n', \dots, c_0'\}$ as sets and $d_{m+n+1} \geq d_{m+n} \geq \dots \geq d_0$.\\

Here is a restatement of the characterization of nilpotent orbits for $Sp(2n,\bb{C})$, which is implicit in the construction of nilpotent orbit closures in \cite{KP 1982}: Any nilpotent orbit for $Sp(2n,\mathbb{C})$ can be parameterized by a partition of $2n$ with column sizes $(c_{2k}, c_{2k-1}, \dots, c_0)$, where $c_{2k} \geq c_{2k-1} \geq \dots \geq c_{0} \geq 0$ (by insisting $c_{2k}$ to be the longest column, we put $c_0 = 0$ if necessary), such that $c_{2i} + c_{2i-1}$ is even for all $i \geq 0$ (where $c_{-1} = 0$).\\

For most parts of the following work, we study the ring of regular functions $R(\mc{O})$ for $\mc{O} = (2a_{2k}, \dots, 2a_1, 2a_0)$, i.e. all columns of $\mc{O}$ are even. For example, the following is known to be true:\\

\begin{theorem}[Barbasch - \cite{B 2008} and \cite{B 2015}, p.29] \label{thm:b2008}
Let $\mc{O} = (2a_{2k}, \dots, 2a_1, 2a_0)$ be a nilpotent orbit such that $a_{2i-1} > a_{2i-2}$ for all $i$. Then as $G \cong K_{\bb{C}}$-modules, the spherical unipotent representation $X_{\mathrm{triv}}$ attached to $\mc{O}$ satisfies
$$X_{\mathrm{triv}} \cong R(\mc{O}).$$
\end{theorem}

%Following the arguments in \cite{B 2008}, we present a proof of Theorem \ref{thm:b2008} in our setting (see Section 3).
In fact, Barbasch in \cite{B 2008} proved a much more general statement than Theorem \ref{thm:b2008} for other classical Lie groups and other unipotent representations. More specifically, the other unipotent representations $X_{\pi}$ (see Equation (\ref{eq:unichar})) corresponds to the global sections of some $G$-equivariant vector bundle $G {\times}_{G_e} V_{\pi}$ of $\mc{O}$. This essentially verifies a conjecture of Vogan (Conjecture 12.1 of \cite{V2}) for such orbits. More details are given in Remarks \ref{rmk:quantization}.\\

With Theorem \ref{thm:b2008}, we can essentially compute the multiplicity of any irreducible representations appearing in $R(\mc{O})$. In particular, we focus on the fundamental representations of $G = Sp(2n,\bb{C})$, given by
$$\mu_i := \wedge^{i}\bb{C}^{2n}/\wedge^{i-2}\bb{C}^{2n}$$
for $i = 1, 2, \dots, n$ (if $i-2 < 0$, take $\wedge^{i-2}\bb{C}^{2n} = \mathrm{triv}$). We have the following formula for the multiplicities of fundamental representations for a larger class of nilpotent orbits than in Theorem \ref{thm:b2008}:\\

\noindent \textbf{Theorem A.}\ \textit{
Let $\mc{O} = (2a_{2k}, \dots, 2a_1, 2a_0)$ be a nilpotent orbit for $G$. Remove all column pairs of same size $(\alpha_i, \alpha_i)$ in $\mc{O}$, leaving the orbit $(d_{2l},d_{2l-1},$ $\dots,d_0)$, i.e.
$$\mc{O} = (d_{2l},d_{2l-1},\dots,d_0) \vee (\alpha_1, \alpha_1, \dots, \alpha_x, \alpha_x)$$
with $d_{i+1} \neq d_i$ for all $i$. Then the multiplicities of the fundamental representations $\mu_i$ are given by
$$[R(\mc{O}):\mu_i] = [Ind_{GL(D)}^{Sp(2n,\bb{C})}(\mathrm{triv}):\mu_i],$$
where $GL(D) = \Pi_{i=0}^l GL(\frac{d_{2i}+d_{2i-1}}{2}) \times \Pi_{i=1}^x GL(\alpha_i)$.}\\

For example, let $\mc{O} = (8,4,4,4,2,2,2) = (8,4,2) \vee (4,4,2,2)$. So
$$[R(\mc{O}) : \mu_i] = [Ind_{GL(6) \times GL(1) \times GL(4) \times GL(2)}^{Sp(26)}(\mathrm{triv}):\mu_i]$$
for all $i = 0, \dots, 13$. More examples can be found in Example \ref{eg:7531}.\\

The second main theorem concerns about the Zariski closure $\overline{\mc{O}}$ of $\mc{O}$. In particular, we are interested in the normality of $\overline{\mc{O}}$. It is known in \cite{KP 1979} that in the case of $SL(n,\bb{C})$, all nilpotent orbit closures are normal. For $Sp(2n,\bb{C})$, Kraft and Procesi proved the following:
\begin{theorem}[Kraft-Procesi - \cite{KP 1982}] \label{thm:KP}
Let $\mc{O} = (c_{2k}, c_{2k-1}, \dots, c_0)$ be a nilpotent orbit for $G = Sp(2n,\mathbb{C})$. If there is a chain of column lengths of the form
$$c_{2i} \neq c_{2i-1} = c_{2i-2} = \dots = c_{2j - 1} = c_{2j-2} \neq c_{2j-3}, $$
then $\overline{\mc{O}}$ is not normal.
\end{theorem}
For instance, for $G = Sp(2n,\bb{C})$, the orbit closures $\overline{(8,6,6,6)}$, $\overline{(6,6,6,6)}$ are normal, while $\overline{(8,6,6,4)}$ is not normal. On the other hand, there is an algebro-geometric criterion of normality. Namely, Proposition 8.2 of \cite{AO 2004} says
$$R(\overline{\mc{O}}) \cong R(\mc{O}) \text{  if and only if  }\overline{\mc{O}}\text{  is normal.}$$
Our second Theorem gives a more refined criterion on normality of $\overline{\mc{O}}$:\\

\noindent \textbf{Theorem B.} \textit{
Let $\mc{O} = (2a_{2k}, \dots, 2a_1, 2a_0)$ be a nilpotent orbit for $Sp(2n,\bb{C})$. Then $\overline{\mc{O}}$ is not normal iff there exists a fundamental representation $\mu_i$ such that}
$$[R(\overline{\mc{O}}) : \mu_i] < [R(\mc{O}):\mu_i].$$

\section{Fundamental Group and Lusztig's quotient of Nilpotent Orbits}
In this Section, we focus on the structure of nilpotent orbits for $G = Sp(2n,\bb{C})$. More precisely, we compute the $G$-\textit{equivariant fundamental group} $A(\mc{O})$ and the \textit{Lusztig's quotient} $\overline{A}(\mc{O})$ of a nilpotent orbit for $G$. All the materials presented in this Section can be found in \cite{CM}, \cite{S1}, \cite{S2}.\\

Let $\mc{O}$ be a nilpotent orbit for $G$, and $e \in \mc{O}$. Then the stabilizer group $G_e$ are the elements in $G$ keeping $e$ fixed, i.e. $G_e = \{g \in G | g\cdot e = e \}$. Following \cite{CM}, define the $G$-\textbf{equivariant fundamental group} of $\mc{O}$ as
$$A(\mc{O}) := G_e/(G_e)^0,$$
where $H^0$ is the identity component of a group $H$. The calculation of $A(\mc{O})$ for $G = Sp(2n,\bb{C})$ can be tracked easily from \cite[Chapter 5]{CM} or \cite{S1} as follows:
\begin{proposition} \label{prop:AO} \mbox{}\\
Let $G = Sp(2n,\bb{C})$ and $\mc{O}$ is a nilpotent orbit for $G$. Then $A(\mc{O}) = (\bb{Z}/2\bb{Z})^b$, where $b$ is the number of distinct even elements in the partition of $\mc{O}$.\\
\end{proposition}
To cater for our forthcoming calculations, it is desirable to express $A(\mc{O})$ using \textit{dual} partitions:
\begin{proposition} \label{prop:AOcol} \mbox{}\\
Let $\mc{O} = (c_{2k},\dots,c_1,c_0)$ be a nilpotent orbit in $Sp(2n,\bb{C})$. Remove all column pairs $c_{2i-1} = c_{2i-2}$, and let $(d_{2j},\dots,d_2,d_1,d_0)$ be the remaining columns, i.e.
$$\mc{O} = (c_{2k},\dots,c_1,c_0) = (d_{2j},\dots,d_2,d_1,d_0) \vee (\nu_1,\nu_1, \dots, \nu_y, \nu_y),$$
with $d_{2i-1} \neq d_{2i-2}$ for all $i$. Then $A(\mc{O}) \cong (\bb{Z}/2\bb{Z})^j$ is generated by $\{s_{2j-1},s_{2j-3},\dots,s_{3},s_1\}$.\\
\end{proposition}
For example, $(8,6,6,4,4,2,2)$ has trivial fundamental group, while $(8,8,6,6,4,4,2,2)$ has fundamental group isomorphic to $(\bb{Z}/2\bb{Z})^4$.
\begin{proof}
Let
$$\mc{O} = \mc{O}^* \vee (\nu_1,\nu_1, \dots, \nu_y, \nu_y),$$
with $\mc{O}^* = (d_{2j},\dots,d_2,d_1,d_0)$ as in the Proposition. Then $A(\mc{O})$ and $A(\mc{O}^*)$ have the same size - Indeed, by removing $(c_{2i-1},c_{2i-2}) = (\nu,\nu)$ from $\mc{O}$, the partition description of $\mc{O}$ changes in the form:
$$[r_1 \geq \dots \geq r_{\nu-1} \geq r_{\nu} > r_{\nu+1} \geq \dots] \to [r_1-2 \geq \dots \geq r_{\nu-1}-2 \geq r_{\nu}-2 \geq r_{\nu+1} \geq \dots]$$
with $r_{\nu} > r_{\nu+1}+2$, or $r_{\nu} = r_{\nu+1}+2$ with $r_{\nu} \equiv r_{\nu+1} \equiv 1(\mathrm{mod }2)$. In both cases, the number of distinct even numbers before and after removal are equal. Hence Proposition \ref{prop:AO} says the $G$-equivariant fundamental groups are the same.\\

We now focus on finding $A(\mc{O}^*)$. Suppose the partition description of $\mc{O}^*$ is $[r_1^* \geq r_2^* \geq \dots \geq r_{d_{2j}}^* > 0]$. Then $r_i^*$ must satisfy the following:\\
$\bullet$ If one of $r_i^*$ and $r_{i+1}^*$ is odd, $r_i^* - r_{i+1}^* =0$ or $1$;\\
$\bullet$ If both $r_i^*$ and $r_{i+1}^*$ are even, $r_i^* - r_{i+1}^* = 0$ or $2$; and\\
$\bullet$ $r_{d_{2j}}^* \leq 2$.\\
Therefore the number of distinct even elements in $\mc{O}^*$ is equal to the number of even elements in $\{1,2,\dots,r_1^*\}$, which is $\lceil \frac{r_1^*}{2} \rceil = j$. Consequently $A(\mc{O}) \cong A(\mc{O}^*) = (\bb{Z}/2\bb{Z})^j$.
\end{proof}

There is a quotient group $\overline{A}(\mc{O})$ of $A(\mc{O})$, first introduced by Lusztig, that plays a vital role in the theory of unipotent representations. We will follow \cite{S2} to determine $\overline{A}(\mc{O})$ under our dual partition notation.

\begin{proposition} \label{prop:barAOcol}
Let $\mc{O} = (c_{2k},\dots,c_1,c_0)$ be a nilpotent orbit for $Sp(2n,\bb{C})$. Remove all column pairs $(c_{2i-1}, c_{2i-2}) = (\nu, \nu)$, along with all \textbf{odd} columns $c_{j} = 2\mu + 1$. Let $(d_{2p}',\dots,d_2',d_1',d_0')$ be the remaining columns, i.e.
$$\mc{O} =  (c_{2n},\dots,c_1,c_0) = (d_{2p}',\dots,d_1',d_0') \vee (2\mu_1+1,\dots,2\mu_x+1) \vee (\nu_1,\nu_1,\dots,\nu_y,\nu_y),$$
then $\overline{A}(\mc{O}) \cong (\bb{Z}/2\bb{Z})^p$ with generators $\{s_{2p-1}',s_{2p-3}'\dots,s_1'\}$.
\end{proposition}

\begin{remark} \mbox{}\\ \emph{
(1)\ Note that the above construction automatically makes $\overline{A}(\mc{O})$ into a quotient of $A(\mc{O})$. Indeed, $\{d_{2p}', \dots, d_1', d_0'\}$ is obtained by removing the odd columns in $\{d_{2j}, \dots, d_1, d_0\}$, so the Lusztig's quotient map
$$p: A(\mc{O}) \to \overline{A}(\mc{O})$$
has kernel $\ker p = \{s_{2l - 1}\ |\ d_{2l-1} \text{ is odd } \}$.}\\

\noindent \emph{(2)\ We will focus on the case when $\mc{O} = (2a_{2k}, \dots, 2a_1, 2a_0)$ in the following sections. In this case, $A(\mc{O})$ is always equal to $\overline{A}(\mc{O})$.}
\end{remark}

\begin{proof}
There is a description of $\overline{A}(\mc{O})$ on \cite[Section 5]{S2} which we will use here. Write $\mc{O} = [r_1 \geq r_2 \geq \dots ]$ in its partition description. Let $S_{odd}$ be the collection of even $r_i$ appearing odd number of times in $\mc{O}$ and similarly for $S_{even}$. Write
\begin{align*}
S_{odd} &= \{ 2\alpha_{2m} > 2\alpha_{2m-1} > \dots > 2\alpha_1 \}\ \text{,putting }\alpha_1 = 0\text{ if necessary};\\
S_{even} &= \{ 2\beta_l > 2\beta_{l-1} > \dots > 2\beta_1 \},
\end{align*}
then $\overline{A}(\mc{O})$ is generated by $x_i$, with $i$ equal to $2\alpha_{2r-1}$ and all $2\beta_s$ lying between $2\alpha_{2r+1} \geq 2\beta_s \geq 2\alpha_{2r}$ for all $r$.\\
\indent It is obvious that the above description of $\overline{A}(\mc{O})$ only depends on the ordering between the even $r_i$'s appearing in $\mc{O}$, but not on the value of $r_i$. So we can reduce our study to $\mc{O}^* = [r_1^* \geq r_2^* \geq \dots]$ as in Proposition \ref{prop:AOcol}. Also, the odd $r_i^*$ does not show up in the calculation of $\overline{A}(\mc{O}^*)$, so we can remove all odd $r_i^*$'s in $\mc{O}^*$ and get
$$\mc{O}^* = (d_{2k},d_{2k-1},\dots, d_2, d_1, d_0) \to \mc{O}^{**} =([d_{2k-1} - \sum_{j=1}^{k-1}(d_{2j} - d_{2j-1})]^2,\dots, [d_3 - (d_2 - d_1)]^2, d_1^2, 0),$$
where all $d_{2i-1} - \sum_{j=1}^{i-1}(d_{2j} - d_{2j-1})$, $1 \leq i \leq k$ are distinct integers with the property $\overline{A}(\mc{O}) = \overline{A}(\mc{O}^*) = \overline{A}(\mc{O}^{**})$. Since $(d_{2j} - d_{2j-1})$ is even for all $j$, the $i$-th columns of $\mc{O}^*$ and $\mc{O}^{**}$ are of the same parity. So it suffices to prove the Proposition for $\mc{O}^{**}$.\\

Consider all orbits with partition and dual partition of the form
$$\mc{Q} = [(2k)^{m_k} > (2k-2)^{m_{k-1}} > \dots > 2^{m_1}] =(v_{2k-1}^2,\dots,v_{2i-1}^2,\dots, v_1^2, 0)$$
with all $v_{2i-1}$ distinct (so $\mc{Q} = \mc{O}^{**}$ is an example). Note that
\begin{equation} \label{eq:barao}
m_k = v_1\ \text{ and } m_{k-j} = v_{2j+1} - v_{2j-1} \text{ for all } j > 0.
\end{equation}
\textbf{Claim:} $\overline{A}(\mc{Q}) = (\bb{Z}/2\bb{Z})^p$, where $p$ is the number of even $v_{2i-1}$'s in $\overline{A}(\mc{Q})$.\\
The proof of the Claim is combinatorial. We give a sketch proof here. Suppose $v_1, v_3, \dots, v_{2x-1}$ are all even and $v_{2x+1}$ is odd. Then Equation (\ref{eq:barao}) says $m_k, \dots, m_{k-x+1}$ are all even and $m_{k-x}$ is odd, i.e. $2k, \dots, (2k-2x+2) \in S_{even}$ and $(2k-2x) = 2\alpha_{2m} \in S_{odd}$. According to the formulation of $\overline{A}(\mc{O})$ in the beginning of the proof,
$$(2k), (2k-2), \dots, (2k-2x+2) > 2\alpha_{2m}$$
and they contribute to $\overline{A}(\mc{Q})$, while $(2k-2x) = 2\alpha_{2m}$ does not contribute to $\overline{A}(\mc{Q})$. So the even columns contribute while the odd column does not contribute to $\overline{A}(\mc{Q})$. This matches with our Claim.\\
Now suppose  $v_{2x+1}, \dots, v_{2y-1}$ are all odd and $v_{2y+1}$ is even. Then Equation (\ref{eq:barao}) says that $(2k-2x-2), \dots, (2k-2y+2) \in S_{even}$ and $(2k-2y )= 2\alpha_{2m-1} \in S_{odd}$. According to the formulation in the beginning of the proof,
$$2\alpha_{2m} > (2k-2x-2), \dots, (2k-2y+2) > 2\alpha_{2m-1}$$
and they do not contribute to $\overline{A}(\mc{Q})$, while $(2k-2y) = 2\alpha_{2m-1}$ contributes to $\overline{A}(\mc{Q})$. So the odd columns do not contribute while the even column contributes to $\overline{A}(\mc{Q})$. This also matches with our Claim.\\

One can continue the argument using induction to prove the Claim holds for all such $\mc{Q}$. In other words, $\overline{A}(\mc{O}^{**})$ is generated by its even column pairs, so the Proposition holds for $\mc{O}^{**}$ and hence for $\mc{O}$ as well.
\end{proof}

\section{special unipotent representations}
Recall the construction of special unipotent representations attached to a special classical nilpotent orbit $\mc{O}$ with $^L\mc{O}$ is an even orbit in \cite{BV 1985} (or \cite{W 2016}).
\begin{algorithm} \label{alg:unipotent}
Let $\mc{O} = (2a_{2k},\dots,2a_1,2a_0)$ with $a_{2i-1} > a_{2i-2}$ for all $i$:\\
(I) The Spaltenstein dual is given by $^L\mc{O} = \bigcup_{i=1}^{k} [2a_{2i}+1, 2a_{2i-1} - 1] \cup [2a_0 + 1]$. Hence
$$\lambda_{\mc{O}} = \frac{1}{2}\ ^Lh = (\lambda_1 ; \dots ; \lambda_k ; a_0, \dots, 2, 1),$$
where $\lambda_i = (a_{2i}, \dots, 2, 1 ;  a_{2i-1} - 1, \dots, 1, 0)$ for each $i$.\\

\noindent (II) Let $\gamma(\mc{O}) := \{\mc{O}' \subseteq \overline{\mc{O}} | \mc{O}' \nsubseteq \overline{\mc{O}}_{spec}\ \text{for any other special orbit } \mc{O}_{spec} \subsetneq \mc{O} \}$. Lusztig in \cite{Lu2} defined an injection $\gamma(\mc{O}) \hookrightarrow \overline{A}(\mc{O})$, such that the composition of maps
$$\sigma: \gamma(\mc{O}) \hookrightarrow \overline{A}(\mc{O}) \stackrel{\text{left cell}} \longrightarrow \hat{W}$$
maps $\mc{O}'$ to its Springer representation. In our case, $\gamma(\mc{O})$ is given by
$$\gamma(\mc{O}) = \bigcup_{I \subset \{s_1,s_3, \dots, s_{2k-1} \}}\{ \mc{O}_{I} \},$$
with
$$\mc{O}_{I} = \bigvee_{s_{2j-1} \notin I} (2a_{2j}, 2a_{2j-1}) \vee \bigvee_{s_{2i-1} \in I} (2a_{2i}+1, 2a_{2i-1}-1) \vee (d_0).$$
By Proposition \ref{prop:barAOcol}, $\overline{A}(\mc{O})$ has $2^k$ elements, which has the same cardinality as $\gamma(\mc{O})$. So the injection $\gamma(\mc{O}) \hookrightarrow \overline{A}(\mc{O})$ is indeed a bijection.\\

From now on, we denote elements in $\overline{A}(\mc{O})$ by the subset $I \subset \{s_1, s_3, \dots, s_{2k-1} \}$. According to the algorithm of computing Springer representations given in \cite[Section 7]{S1}, $\sigma(\mc{O}_I) = j_{W_I}^{W} (sgn)$ where
$$W_I = \prod_{s_{2j-1} \notin I} (C_{a_{2j}} \times D_{a_{2j-1}}) \times \prod_{s_{2i-1} \in I} (D_{a_{2i}+1} \times C_{a_{2i-1}-1}) \times C_{a_0}.$$

\noindent (III) Let $\lambda_j = (a_{2j}, \dots, 2, 1 ;  a_{2j-1} - 1, \dots, 1, 0)$ as before, and $\lambda_i = (a_{2i}, \dots, 1,0 ;$ $a_{2i-1}-1, \dots, 2, 1)$. Define

$$R_I = \sum_{w \in W_I}(-1)^{l(w)} X\left(
\begin{array}{cccccccccc}
& \cup_{s_{2j-1} \notin I} \lambda_j; & \cup_{s_{2i-1} \in I} \lambda_i ; & a_0 \dots 2,1 \\
w(& \cup_{s_{2j-1} \notin I} \lambda_j; & \cup_{s_{2i-1} \in I} \lambda_i ; & a_0 \dots 2,1)\end{array}
\right),$$
where
$X \left(
\begin{array}{cccccccccc}
\lambda_1 \\
\lambda_2 \\
\end{array}
\right) = K-\text{finite part of }Ind_B^G(e^{(\lambda_1,\lambda_2)} \otimes 1)$
is the \textbf{principal series representation} with character $(\lambda_1,\lambda_2) \in \mf{h}_{\bb{C}}$, the complexification of the maximal torus $\mf{h}$ in $\mf{g}$ (here we treat $G$ as a real Lie group). In particular, the $G \cong K_{\bb{C}}$-types of $X \left(
\begin{array}{cccccccccc}
\lambda_1 \\
\lambda_2 \\
\end{array}
\right)$ is equal to $Ind_T^G(e^{\lambda_1-\lambda_2})$, which we will denote as $Ind_T^G(\lambda_1-\lambda_2)$ subsequently (see Theorem 1.8 of \cite{BV 1985} for more details on the principal series representations).\\

\noindent (IV) The special unipotent representations are parameterized by $\pi \in \overline{A}(\mc{O})^{\wedge}$. In fact, for any $I \subset \{s_1, \dots, s_{2k-1}\}$, there exists an irreducible $\overline{A}(\mc{O})$-representation $\pi_I$ with
$$
    \pi_I(s_{2j-1})=
\begin{cases}
    -1 ,& \text{if } s_{2j-1} \in I\\
    1,              & \text{otherwise}.
\end{cases}
$$
All $\pi \in \overline{A}(\mc{O})^{\wedge}$ can be obtained in this way, i.e. $\pi = \pi_I$ for some $I$. Then the special unipotent representations are of the form:
\begin{equation} \label{eq:unichar}
X_{\pi_I} = \frac{1}{2^k}\sum_{J \subset \{s_1, s_3, \dots s_{2k-1} \}} \mathrm{tr}_{\pi_{I}}(J)R_J,
\end{equation}
where $\pi_I(J) = \Pi_{s_{2j-1} \in J}\pi_I(s_{2j-1})$. In particular,
\begin{equation} \label{eq:unitriv}
X_{\mathrm{triv}} = X_{\pi_{\phi}} = \frac{1}{2^k}\sum_{J \subset \{s_1, s_3, \dots s_{2k-1} \}}R_J
\end{equation}
and the sum of all special unipotent representations is given by
\begin{equation} \label{eq:unisum}
R_{\phi} = \bigoplus_{\pi \in \overline{A}(\mc{O})^{\wedge}} X_{\pi}
\end{equation}
\end{algorithm}

\begin{example} \label{eg:8642} \emph{
Let $\mc{O}' = [4,4,3,3,2,2,1,1] = (8,6,4,2)$. Then $A(\mc{O}) =$ $\overline{A}(\mc{O}) =$ $\{\phi,$ $\{s_1\},$ $\{s_3\},$ $\{s_1,s_3\}\}$, and above calculation gives
$$\lambda_{\mc{O}} = (4321; 210; 21; 0),$$
with
\begin{align*}
R_{\phi} &= \sum_{w \in C_4 \times D_3 \times C_2 \times D_1}(-1)^{l(w)} X\left(
\begin{array}{cccccccccc}
&4321; & 210; & 21; & 0 \\
w(&4321; & 210; & 21; & 0)\end{array}
\right); \\
R_{s_1} &= \sum_{w \in C_4 \times D_3 \times D_3 }(-1)^{l(w)} X\left(
\begin{array}{cccccccccc}
&4321; & 210; & 210 \\
w(&4321; & 210; & 210)\end{array}
\right); \\
R_{s_3} &= \sum_{w \in D_5 \times C_2 \times C_2 \times D_1}(-1)^{l(w)} X\left(
\begin{array}{cccccccccc}
&43210; & 21; & 21; & 0 \\
w(&43210; & 21; & 21; & 0)\end{array}
\right); \\
R_{s_1,s_3} &= \sum_{w \in D_5 \times C_2 \times D_3}(-1)^{l(w)} X\left(
\begin{array}{cccccccccc}
&43210; & 21; & 210 \\
w(&43210; & 21; & 210)\end{array}
\right).
\end{align*}
and $X_{\mathrm{triv}} = \frac{1}{4}(R_{\phi} + R_{s_1} + R_{s_3} + R_{s_1, s_3})$.}
\end{example}

\begin{remark} \label{rmk:quantization} \emph{
According to Theorem \ref{thm:b2008}, we have
$$X_{\mathrm{triv}} \cong R(\mc{O}) \cong Ind_{G_e}^G(\mathrm{triv}).$$
As a generalization of Theorem \ref{thm:b2008}, by the last paragraph of \cite{B 2008}, or more explicitly, p.29 of \cite{B 2015}, it can be seen that as $K_{\bb{C}} \cong G$-modules,
$$X_{\pi} \cong Ind_{G_e}^G(\pi)$$
for all local systems $\pi \in A(\mc{O})^{\wedge} = (G_e/(G_e)^0)^{\wedge}$. In other words, we have attached a unitary representation $X_{\pi}$ (by \cite{B 1989}) to each \textbf{orbit data} $(\mc{O}, \pi)$ for all $\mc{O}$ satisfying the hypothesis of Theorem \ref{thm:b2008} (for more details on orbit data, see \cite{V2} or Section 2 of \cite{AHV}).}\\

\emph{More generally, we can further extend our scheme to a larger class of nilpotent orbits - Consider nilpotent orbits $\mc{O} = (2a_{2k}, \dots, 2a_1, 2a_0)$ with no restrictions on the size of columns. Separate the columns $2a_{2i-1} = 2a_{2i-2}$ from $\mc{O}$, i.e.
$$\mc{O} = \mc{O}' \vee (\nu_1, \nu_1, \dots, \nu_y, \nu_y),$$
where $\mc{O}'$ satisfies the hypothesis of Theorem \ref{thm:b2008}. Then $\mc{O} = Ind_{\mf{g}' \times \mf{gl}(\nu_1) \times \dots \times \mf{gl}(\nu_y)}^{\mf{g}}(\mc{O}' \oplus \mathrm{triv} \oplus \dots \oplus \mathrm{triv})$. By Proposition \ref{prop:AOcol}, $A(\mc{O}) = A(\mc{O}')$ and hence there is a natural one-to-one correspondence between $\Pi \in A(\mc{O})^{\wedge}$ and $\pi \in A(\mc{O}')^{\wedge}$. By Corollary 1.3 of \cite{LZ},
$$ Ind_{G' \times GL(\nu_1) \times \dots \times GL(\nu_y)}^G( Ind_{G_e'}^{G'}(\pi) \boxtimes \mathrm{triv} \boxtimes \dots \boxtimes \mathrm{triv}) = Ind_{G_e}^G(\Pi).$$
Therefore, the unitarily induced module $Ind_{G' \times GL(\nu_1) \times \dots \times GL(\nu_y)}^G( X_{\pi} \boxtimes \mathrm{triv} \boxtimes \dots \boxtimes \mathrm{triv})$ is the corresponding unitary representation attached to the orbit data $(\mc{O}, \Pi)$.}
\end{remark}

Using the formula of $X_{\pi}$ and the techniques in Proposition 4.2-4.3 of \cite{W 2016}, we can compute the \textbf{Lusztig-Vogan bijection} $\gamma(\mc{O},\pi)$ (Section 1 of \cite{W 2016}) for all local systems $\pi$ of all $\mc{O}$'s discussed in Remark \ref{rmk:quantization}. For example, let $\mc{O} = (8,6,4,2)$, then
\begin{align*}
\gamma(\mc{O},\pi_{\phi}) &= (8,6,4,4,4,2,2,2,0,0) \sim (8,6,4,4,2,2,0;4,2,0)\\
\gamma(\mc{O},\pi_{s_1}) &= (8,6,4,4,4,2,2,1,1,0) \sim (8,6,4,4,2,2,0;4,1,1)\\
\gamma(\mc{O},\pi_{s_3}) &= (8,5,5,4,3,3,2,1,1,0) \sim (8,5,5,3,3,1,1;4,2,0)\\
\gamma(\mc{O},\pi_{s_1, s_3}) &= (8,5,5,4,3,3,1,1,1,1) \sim (8,5,5,3,3,1,1;4,1,1)
\end{align*}
Moreover, Theorem 5.1 of \cite{W 2016} can be readily verified as well.

\section{Fundamental Multiplicities}
By Theorem \ref{thm:b2008} and the character formula of $X_{\mathrm{triv}}$ given in Algorithm \ref{alg:unipotent}, one can practically compute the multiplicities of any irreducible $G$-representations appearing in $R(\mc{O})$. In this Section, we focus on the fundamental multiplicities $\mu_i = V_{1^i 0^{n-i}}$ defined in the Introduction (the subscript $\beta$ of $V_{\beta}$ denotes the highest weight of an irreducible representation of $G$).
\begin{lemma} \label{lem:uab}
For $G = Sp(2n,\bb{C})$ and $2a \geq 2b > 0$ both even, define the virtual $G$-modules
\begin{align*}
U_{a,b} &=  \frac{1}{2}[\sum_{w \in C_{a} \times D_{b}}(-1)^{l(w)} Ind_T^G( (a, \dots, 2,1; b-1, \dots, 1,0) - w(a, \dots, 2,1; b-1, \dots, 1,0))\\
 &\ \ + \sum_{w \in D_{a+1} \times C_{b-1}}(-1)^{l(w)} Ind_T^G( (a, \dots, 1,0 ; b-1, \dots, 2,1) - w(a, \dots, 1,0 ; b-1, \dots, 2,1))];\\
U_{a} &=  \sum_{w \in C_{a}}(-1)^{l(w)} Ind_T^G(a, \dots, 2,1 - w(a, \dots, 2,1)),
\end{align*}
then $[U_{a,b}: \mu_i] = [Ind_{GL(a+b)}^G(\mathrm{triv}): \mu_i] = \delta_{i0},\ [U_{a}: \mu_i] = [Ind_{GL(a)}^G(\mathrm{triv}): \mu_i] = \delta_{i0}$ for all $i$ (where $\delta_{ij}$ is the Kronecker delta function).
\end{lemma}

\begin{proof}
Note that $U_{a,b}$ and $U_a$ are character formulas of the special unipotent representation $X_{\mathrm{triv}}$, for nilpotent orbits $\mc{O} = (2^{2b}1^{2a-2b})$ and $\mathrm{triv} = (1^{2n})$ respectively. It is proved in Section 2 of \cite{W 2016} that for all such $\mc{O}$, all $G$-representations of $R(\mc{O})$ are of the form $V_{2p_1, 2p_2, \dots, 2p_b, 0,\dots,0}$ with multiplicity one. In particular, no $\mu_i = V_{1^i 0^{n-i}}$ appears in $R(\mc{O})$ for $i > 0$. Hence $[U_{a,b}: \mu_i] = [U_{a}: \mu_i] = 0$ for all $i > 0$.\\

On the other hand, one can use Frobenius reciprocity to conclude that $[Ind_{GL(a+b)}^G(\mathrm{triv}): \mu_i]= [Ind_{GL(a)}^G(\mathrm{triv}): \mu_i] = 0$ for all $i > 0$. Hence the result follows.
\end{proof}

\begin{lemma} \label{lem:decompose}
Let $G = Sp(2p+2q,\bb{C})$, with $G_1 = Sp(2p,\bb{C})$ and $G_2 = Sp(2q,\bb{C})$ be subgroups of $G$ such that $G_1 \times G_2$ embeds into $G$ diagonally. Write $T_1$, $T_2$ as Cartan subgroups of $G_1$ and $G_2$ respectively (so that $T := T_1 \times T_2$ is a Cartan subgroup of $G$, and let
$$C_1 = \sum_i a_i Ind_{T_1}^{G_1}(\gamma_i),\ \ C_2 = \sum_j b_j Ind_{T_2}^{G_2}(\delta_j)$$
be virtual representations of $G_1$ and $G_2$ respectively. Then as virtual $G$-modules,
$$\sum_{i,j}a_ib_jInd_T^G(\gamma_i;\delta_j) \cong Ind_{G_1 \times G_2}^G(C_1 \boxtimes C_2)$$
\end{lemma}
\begin{proof}
Suppose $\mu$ is a finite dimensional, irreducible representation of $G$, writing $\mu|_{G_1 \times G_2} = \oplus_k (\pi_1^k \boxtimes \pi_2^k)$ as the restricted representation to $G_1 \times G_2$, then
\begin{align*}
[Ind_{G_1 \times G_2}^G(C_1 \boxtimes C_2): \mu] &= [C_1 \boxtimes C_2 : \mu|_{G_1 \times G_2}]\\
&=\sum_k[\sum_i a_i Ind_{T_1}^{G_1}(\gamma_i) : \pi_1^k][\sum_j b_j Ind_{T_2}^{G_2}(\delta_j): \pi_2^k]\\
&=\sum_{i,j,k}a_ib_j[\gamma_i : \pi_1^k|_{T_1}][\delta_j : \pi_2^k|_{T_2}]\\
&=\sum_{i,j,k}a_ib_j[(\gamma_i;\delta_j) : (\pi_1^k \boxtimes \pi_2^k)|_{T_1 \times T_2}]\\
&=\sum_{i,j}a_ib_j[(\gamma_i;\delta_j) : \mu|_{T_1 \times T_2}]\\
&=\sum_{i,j}a_ib_j[Ind_T^G(\gamma_i;\delta_j):\mu]
\end{align*}
\end{proof}

\begin{proposition} \label{prop:thmb1}
Let $\mc{O}' = (2a_{2p}',\dots, 2a_1', 2a_0')$ be a nilpotent orbit with $a_{2i-1} > a_{2i-2}$ for all $i$. Then the multiplicities of the fundamental representations are given by
$$[R(\mc{O}'): \mu_i] = [Ind_{GL(D')}^G(\mathrm{triv}): \mu_i],$$
with $GL(D') = \Pi_{i=0}^p GL(a_{2i}' + a_{2i-1}')$ (Recall we take $a_{-1}' = 0$ in the Introduction).
\end{proposition}
\begin{remark} \emph{
The above Proposition essentially shows Theorem A holds for all $\mc{O}'$'s with $a_{2i-1}' > a_{2i-2}'$ for all $i$ - According to Theorem A, one needs to remove all column pairs of the same size. By the hypothesis of the above Proposition, a column pair can only exist when $(\alpha, \alpha) = (2a_{2i}', 2a_{2i-1}')$. Therefore, $GL(a_{2i}' + a_{2i-1}') = GL(2a_{2i}') = GL(\alpha)$ as in Theorem A.}
\end{remark}
\begin{proof}
We will prove the Theorem by induction on $p$. Note that $\mc{O}'$ satisfies the hypothesis in Theorem \ref{thm:b2008}, therefore $R(\mc{O}') \cong X_{\mathrm{triv}}$ as $G$-modules. For the rest of the proof, we will use the notation $X_{\pi, \mc{O}}$ to denote the unipotent representation $X_{\pi}$ attached to the orbit $\mc{O}$.\\

\noindent $\mathbf{p = 0}$: $\mc{O}_0 = (d_0')$. According to Algorithm \ref{alg:unipotent}, $R(\mc{O}_0) \cong X_{\mathrm{triv}, \mc{O}_0} \cong U_{a_0'}$. So the result follows from Lemma \ref{lem:uab}.\\

\noindent \textbf{Induction Step:} Suppose the hypothesis holds for $\mc{O}_r = (2a_{2r}',\dots,2a_1',2a_0')$ and $G = G_r = Sp(\sum_{i=0}^{2r}2a_{i}',\bb{C})$, i.e.
$$[R(\mc{O}_r):\mu_i] = [X_{\mathrm{triv},\mc{O}_r}:\mu_i] = [Ind_{GL(D_r')}^G(\mathrm{triv}): \mu_i],$$
where $GL(D_r') = \Pi_{i=0}^r GL(a_{2i}' + a_{2i-1}')$. Now study the orbit $\mc{O}_{r+1} = (2a,2b,2a_{2r}', \dots,d_1',d_0)$ and $G = G_{r+1}$. Algorithm \ref{alg:unipotent} gives
\begin{align*}
X_{\mathrm{triv}, \mc{O}_{r+1}} &= \frac{1}{2^{r+1}}\sum_{I \subset \{s_1', s_3', \dots s_{2r+1}' \}}R_I\\
&= \frac{1}{2^{r+1}}(\sum_{J \subset \{s_1', s_3', \dots s_{2r-1}' \}}R_J + \sum_{J \subset \{s_1', s_3', \dots s_{2r-1}' \}}R_{J\cup \{s_{2r+1}'\}})\\
&\cong \frac{1}{2^{r+1}}[\sum_{J \subset \{s_1', s_3', \dots s_{2r-1}' \}} \sum_{C_{a} \times D_{b} \times W_J}(-1)^{l(w)} Ind_T^G((\lambda_{r+1};\lambda_{\mc{O}_r}) - w(\lambda_{r+1};\lambda_{\mc{O}_r}))\\
&\ \ \ \ +\sum_{J \subset \{s_1', s_3', \dots s_{2r-1}' \}} \sum_{D_{a+1} \times C_{b-1} \times W_J}(-1)^{l(w)} Ind_T^G((\lambda_{r+1}'; \lambda_{\mc{O}_r}) - w( \lambda_{r+1}'; \lambda_{\mc{O}_r}))],
\end{align*}
with $\lambda_{r+1} = (a,\dots,2,1; b-1,\dots, 1, 0)$, $\lambda_{r+1}' = (a,\dots, 1,0; b-1,\dots,2,1)$ and $\lambda_{\mc{O}_r}$ is as in Step (1) of Algorithm \ref{alg:unipotent}. Now apply Lemma \ref{lem:decompose} with $G_1 = Sp(2a+2b)$ and $G_2 = G_r$, and
\begin{align*}
C_1 &= U_{a, b} = \frac{1}{2}[\sum_{C_{a} \times D_{b}}(-1)^{l(w)} Ind_{T_1}^{G_1}(
\lambda_{r+1} - w\lambda_{r+1}) + \sum_{D_{a+1} \times C_{b-1}}(-1)^{l(w)} Ind_{T_1}^{G_1}(\lambda_{r+1}' - w\lambda_{r+1}')];\\
C_2 &= \frac{1}{2^{r}}\sum_{J \subset \{s_1', s_3', \dots s_{2r-1}' \}} \sum_{W_J}(-1)^{l(w)} Ind_{T_2}^{G_2}(\lambda_{\mc{O}_r} - w\lambda_{\mc{O}_r}) \cong X_{\mathrm{triv}, \mc{O}_r} ,
\end{align*}
we will get $X_{\mathrm{triv}, \mc{O}_{r+1}} \cong Ind_{Sp(2a+2b) \times G_r}^G(U_{a,b} \boxtimes X_{\mathrm{triv},\mc{O}_r})$.\\

Now all fundamental representations $\mu_i$ in $G$ decomposes as $\mu_i|_{Sp(2a+2b) \times G_r} = \oplus_{(\gamma,\delta)} \mu_{\gamma}^1 \boxtimes \mu_{\delta}^2$, with all $\mu_{\gamma}^1$ and $\mu_{\delta}^2$ being fundamental representations of $G_1$ and $G_2$ respectively.
\begin{align*}
[R(\mc{O}_{r+1}): \mu_i] = [X_{\mathrm{triv}, \mc{O}_{r+1}} : \mu_i] &= [U_{a,b} \boxtimes X_{\mathrm{triv},\mc{O}_r} : \mu_i|_{Sp(2a+2b) \times G_r}]\\
&=\bigoplus_{(\gamma,\delta)} [U_{a,b} : \mu_{\gamma}^1][X_{\mathrm{triv}, \mc{O}_r} : \mu_{\delta}^2]\\
&=\bigoplus_{(\gamma,\delta)} [Ind_{GL(a+b)}^{Sp(2a+2b)}(\mathrm{triv}) : \mu_{\gamma}^1][Ind_{GL(D_r')}^{G_r}(\mathrm{triv}) : \mu_{\delta}^2]\\
&=[Ind_{GL(a+b)}^{Sp(2a+2b)}(\mathrm{triv})\boxtimes Ind_{GL(D_r')}^{G_r}(\mathrm{triv}) : \mu_i|_{Sp(2a+2b) \times G_r}]\\
&=[Ind_{GL(D_{r+1}')}^G(\mathrm{triv}) : \mu_i],
\end{align*}
where the third line comes from Lemma \ref{lem:uab} the induction hypothesis. So the proof is complete.
\end{proof}

The proof of Theorem A follows immediately from the above Proposition:
\begin{corollary} \label{cor:fundmult}
Let $\mc{O} = (2a_{2k}, \dots, 2a_1, 2a_0)$ be a nilpotent orbit for $G$. Remove all column pairs of same size $(\alpha_i, \alpha_i)$ in $\mc{O}$, leaving the orbit $(d_{2l},d_{2l-1},$ $\dots,d_0)$, i.e.
$$\mc{O} = (d_{2l},d_{2l-1},\dots,d_0) \vee (\alpha_1, \alpha_1, \dots, \alpha_x, \alpha_x)$$
with $d_{i+1} \neq d_i$ for all $i$. Then the multiplicities of the fundamental representations are given by
$$[R(\mc{O}):\mu_i] = [Ind_{GL(D)}^{Sp(2n,\bb{C})}(\mathrm{triv}):\mu_i],$$
where $GL(D) = \Pi_{i=0}^l GL(\frac{d_{2i}+d_{2i-1}}{2}) \times \Pi_{i=1}^x GL(\alpha_i)$.
\end{corollary}
\begin{proof}
We separate the columns $2a_{2i-1} = 2a_{2i-2}$ in $\mc{O}$ as in Remark \ref{rmk:quantization}, i.e. $\mc{O} = \mc{O}' \vee (\nu_1, \nu_1, \dots, \nu_y, \nu_y)$, where $\mc{O}'$ satisfies the hypothesis of \ref{prop:thmb1}. Taking $\pi = \mathrm{triv}$ in Remark \ref{rmk:quantization}, we get
$$R(\mc{O}) = Ind_{G' \times GL(\nu_1) \times \dots \times GL(\nu_y)}^G(R(\mc{O}') \boxtimes \mathrm{triv} \boxtimes \dots \boxtimes \mathrm{triv}).$$
By Proposition \ref{prop:thmb1}, $R(\mc{O}') = Ind_{GL(D')}^{G'}(\mathrm{triv})$, hence the result follows from induction in stages.
\end{proof}

As mentioned in the beginning of this Section, we present an example to compute the multiplicities of irreducible $G$-representations $V_{\beta}$ appearing in $R(\mc{O})$ other than the fundamental representations $V_{1^i0^{n-i}}$ here.

\begin{example} \emph{
Let $\mc{O} = (8,4)$. Then the character formula $X_{\mathrm{triv}} \cong R(\mc{O})$ can be expanded as:}
$$X_{\mathrm{triv}} = \frac{1}{2}[\sum_{C_4 \times D_2}(-1)^{l(w)} Ind_T^G(4321,10 -  w(4321,10)) + \sum_{D_5 \times C_1}(-1)^{l(w')} Ind_T^G(43210,1 -  w'(43210,1))].$$
\emph{To find the coefficient of $Ind_T^G(0^6)$ in the above expression, one needs to find out how many $w \in W(C_4 \times D_2)$ so that $(4321,10)-w(4321,10)$ can be $W$-conjugated to have weight $(0^6)$ (and respectively for $w' \in W(D_5 \times C_1)$). Obviously this forces $w = w' = Id$, and hence }
$$R(\mc{O}) \cong X_{\mathrm{triv}}|_{K_{\bb{C}}} = \frac{1}{2}(Ind_T^G(0^6) + Ind_T^G(0^6)) + \sum_{\lambda \in \mf{t}^*} c_{\lambda} Ind_T^G(\lambda) = Ind_T^G(0^6) + \sum_{\lambda \in \mf{t}^*} c_{\lambda} Ind_T^G(\lambda),\ ||\lambda|| > 0.$$
\emph{To find out the coefficients of $Ind_T^G(1^20^4)$, one needs to find out which $w \in W(C_4 \times D_2)$  so that $(4321,10)-w(4321,10)$ can be $W$-conjugated to have weight $(1^20^4)$ (and respectively for $w' \in W(D_5 \times C_1)$). The list of all such $w(4321,10)$ and $w'(43210,1)$ are given below:}\\
$$\begin{tabular}{|l|l|}
\hline
$w(4321,10)$ & $(4321,10) - w(4321,10)$  \\ \hline
\ \ \ 3421,10 & \ \ \ \ \ \ \ \ \ \ \ \ 1-10000\\
\ \ \ 4231,10 & \ \ \ \ \ \ \ \ \ \ \ \ 01-1000\\
\ \ \ 4312,10 & \ \ \ \ \ \ \ \ \ \ \ \ 001-100\\
\ \ \ 4321,01 & \ \ \ \ \ \ \ \ \ \ \ \ 00001-1\\
\ \ \ 4321,0-1& \ \ \ \ \ \ \ \ \ \ \ \ 000011 \\
\hline
\end{tabular}\ \ \ \ \ \
\begin{tabular}{|l|l|}
\hline
$w'(43210,1)$ & $(43210,1) - w'(43210,1)$  \\ \hline
\ \ \ 34210,1 & \ \ \ \ \ \ \ \ \ \ \ \ 1-10000\\
\ \ \ 42310,1 & \ \ \ \ \ \ \ \ \ \ \ \ 01-1000\\
\ \ \ 43120,1 & \ \ \ \ \ \ \ \ \ \ \ \ 001-100\\
\ \ \ 43201,1 & \ \ \ \ \ \ \ \ \ \ \ \ 0001-10\\
\ \ \ 4320-1,1& \ \ \ \ \ \ \ \ \ \ \ \ 000110 \\
\hline
\end{tabular}$$\\
\emph{Since each of the $w$ and $w'$ above is a simple reflection, i.e. $l(w) = l(w') = 1$, therefore $(-1)^{l(w)} = (-1)^{l(w')} = -1$ and}
$$R(\mc{O}) \cong Ind_T^G(0^6) + \frac{1}{2}[(-5) + (-5)]Ind_T^G(1^20^4) + \sum_{\lambda \in \mf{t}} c_{\lambda} Ind_T^G(\lambda),\ \ ||\lambda||^2 > 2$$
\emph{as virtual $G$-modules. Continuing the calculations, we get}
$$R(\mc{O}) \cong Ind_T^G(0^6) - 5Ind_T^G(1^20^4) + 6Ind_T^G(1^40^2) + 0Ind_T^G(2^10^5) - Ind_T^G(1^6) + \dots$$
\emph{For any irreducible $G$-representation $\mu$, Frobenius reciprocity gives}
$$[R(\mc{O}):\mu] = [(0^6) : \mu|_{T}] - 5[(1^20^4) : \mu|_{T}] +6[(1^40^2) : \mu|_{T}] + 0[(2^10^5) : \mu|_{T}] - [(1^6): \mu|_{T}] + \dots$$
\emph{so in practice this gives $[R(\mc{O}):\mu]$ for any $\mu$ - for example, if $\mu = V_{2^10^5}$, then Weyl character formula gives $\mu|_{T} = 6 \times (0^6) + 1 \times (1^20^4) + 1 \times (2^10^5) + \dots$, with the remaining terms lying outside the dominant Weyl chamber $C = \{(a_1, a_2, \dots, a_6) \in \mf{t}^*\ |\ a_1 \geq a_2 \geq \dots \geq 0 \}$. So }
$$[R(\mc{O}): V_{2^10^5}] = 1 \times 6 - 5 \times 1 + 6 \times 0 + 0 \times 1 - 1\times 0 = 1.$$
\emph{Indeed, since $\mc{O}$ is a spherical orbit, the multiplicity of $V_{2^10^5}$ in $R(\mc{O})$ is known (e.g. Section 2 of \cite{W 2016}) to be one.}
\end{example}

\section{Normality of Orbit Closures}
One of the reasons we are interested in computing the multiplicity of fundamental representations appearing in $R(\mc{O})$ is to detect non-normality of the orbit closure $\overline{\mc{O}}$. To do so, we will find an upper bound on $[R(\overline{\mc{O}}):\mu]$ for all fundamental representations $\mu$, and show that this upper bound is strictly smaller than $[R(\mc{O}):\mu]$ if $\overline{\mc{O}}$ is not normal.\\

The upper bound we need is given in the Lemma below:
\begin{lemma} \label{lem:osharp}
Let $G = Sp(2n,\bb{C})$ and $\mc{O} = (c_{2k},c_{2k-1},\dots, c_{0})$ be \textbf{any} nilpotent orbit. For \textbf{any} finite dimensional irreducible representations $\mu$,
$$[R(\overline{\mc{O}}):\mu] \leq [R(\mc{O}^{\sharp}):\mu]$$
where $\mc{O}^{\sharp} = (\frac{c_{2k}+c_{2k-1}}{2},\frac{c_{2k}+c_{2k-1}}{2}, \frac{c_{2k-2}+c_{2k-3}}{2},\frac{c_{2k-2}+c_{2k-3}}{2},\dots,\frac{c_{2}+c_{1}}{2},\frac{c_{2}+c_{1}}{2},\frac{c_0}{2},\frac{c_0}{2})$.
\end{lemma}
\begin{proof}
We only work in the case when $G = Sp(2n,\bb{C})$. Note that by the Kraft-Procesi criterion (Theorem \ref{thm:KP}), $\overline{\mc{O}^{\sharp}}$ is normal. Therefore $R(\mc{O}^{\sharp}) = R(\overline{\mc{O}^{\sharp}})$. On the other hand, note that $\overline{\mc{O}^{\sharp}} \supset \overline{\mc{O}}$. Consequently, we have a $G$-module surjection
$$R(\mc{O}^{\sharp}) = R(\overline{\mc{O}^{\sharp}}) \twoheadrightarrow R(\overline{\mc{O}})$$
and hence $[R(\overline{\mc{O}}) : \mu] \leq [R(\overline{\mc{O}^{\sharp}}) : \mu]$ for any finite dimensional $G$-representations $\mu$. However, the latter term is equal to $[R(\mc{O}^{\sharp}):\mu]$. Hence the result follows.
\end{proof}
\noindent \textit{Proof of Theorem B.} One direction is easy - if $\overline{\mc{O}}$ is normal, then $R(\overline{\mc{O}}) = R(\mc{O})$ as $G$-modules, hence $[R(\overline{\mc{O}}) : \mu_i] = [R(\mc{O}):\mu_i]$ for all $i$.\\

Now suppose $\mc{O} = (2a_{2k},2a_{2k-1},\dots,2a_1,2a_0)$ be a nilpotent orbit such that $\overline{\mc{O}}$ is not normal. Then Theorem A says
$$[R(\mc{O}) : \mu_i] = [Ind_{GL(D)}^G (\mathrm{triv}) : \mu_i].$$
According to Lemma \ref{lem:osharp}, $\mc{O}^{\sharp} = (a_{2k}+a_{2k-1},a_{2k}+a_{2k-1},\dots,a_{2}+a_{1},a_2+a_1,a_0,a_0)$. Since there may be some odd columns appearing in $\mc{O}^{\sharp}$, we cannot use Theorem A directly to compute $[R(\mc{O}^{\sharp}):\mu_i]$. However, $\mc{O}^{\sharp} = Ind_{\mf{gl}(a_{2k}+a_{2k-1}) \oplus \dots \oplus \mf{gl}(a_2+a_1) \oplus \mf{gl}(a_0)}^{\mf{g}}(\mathrm{triv})$ is strongly Richardson, therefore
$$[R(\mc{O}^{\sharp}) : \mu_i] = [Ind_{GL(D^{\sharp})}^G (\mathrm{triv}) : \mu_i]$$
with $GL(D^{\sharp}) = GL(a_{2k}+a_{2k-1}) \times \dots \times GL(a_2+a_1) \times GL(a_0)$. By the Kraft-Procesi description of non-normal orbit closures (Theorem \ref{thm:KP}), the two parabolic subgroups $GL(D)$ and $GL(D^{\sharp})$ of $G$ are different. More precisely, define

$$F := \{i \in \bb{N}\ |\ GL(i)\text{ is a factor of }GL(D)\};\ \ F^{\sharp} := \{i \in \bb{N}\ |\ GL(i)\text{ is a factor of }GL(D^{\sharp})\}$$
with multiplicities. Rearrange the elements in $F$ and $F^{\sharp}$ so that
$$F = \{f_0 \leq  f_2 \leq  f_4 \dots\};\ \ F^{\sharp} = \{f_0^{\sharp} \leq f_2^{\sharp} \leq f_4^{\sharp} \dots\}.$$
Since $\mc{O}$ has non-normal closure, it has columns of the form
$$2a_{2i} > 2a_{2i-1} = 2a_{2i-2} = \dots = 2a_{2j - 1} = 2a_{2j-2} > 2a_{2j-3} $$
pick the smallest value of such $c_{2j-2} \neq c_{2j-3}$. Then
$$\{ a_0, a_2+a_1, \dots, a_{2j-4} + a_{2j-5} \} = \{f_0, f_2, \dots, f_{2j-4}\} =  \{f_0^{\sharp}, f_2^{\sharp}, \dots, f_{2j-4}^{\sharp}\}$$
while $f_{2j-2}^{\sharp} := a_{2j-2} + a_{2j-3} \in F^{\sharp}$ but not in $F$. More precisely, it is easy to see $f_{2j-2}^{\sharp} < f_{2j-2}$.\\
Let $f := \sum_{i = 0}^{j-1} f_{2i}^{\sharp}$. Using Frobenius reciprocity, one can check that
$$[R(\mc{O}^{\sharp}):\mu_{2f+2}] = [Ind_{\Pi GL(f_{2i}^{\sharp})}^G(\mathrm{triv}) : \mu_{2f+2}]< [Ind_{\Pi GL(f_{2i})}^G(\mathrm{triv}) : \mu_{2f+2}] = [R(\mc{O}):\mu_{2f+2}]$$
Now Lemma \ref{lem:osharp} gives $[R(\overline{\mc{O}}):\mu_i] \leq [R(\mc{O}^{\sharp}):\mu_i]$ for all $i$, and consequently the Theorem follows by taking $i = 2f+2$. \qed

\begin{example} \label{eg:7531} \emph{
Let $\mc{O} = (8,6,6,4,4,2,2)$ for $Sp(32,\mathbb{C})$. Following Theorem \ref{thm:KP}, its closure is not normal.}\\

\emph{Using the notations in the above proof, $F = \{2,4,4,6\}$. Now $\mc{O}^{\sharp} = (7,7,5,5,3,3,1,1)$, and $F^{\sharp} = \{1,3,5,7\}$. The first discrepancy between $F$ and $F^{\sharp}$ occurs at $1 \neq 2$. Hence $[R(\overline{\mc{O}}):\mu_i] < [R(\mc{O}):\mu_i]$ must occur at $i = 2(1) + 2 = 4$. Using Frobenius reciprocity, we computed the multiplicities as follows:
\small{\begin{center}
  \begin{tabular}{ | c || c | c | c | c | c | c | c | c | c | c | c | c | c | c | c | c | c |}
    \hline
    $i$ & 0 & 1 & 2 & 3 & 4 & 5 & 6 & 7 & 8 & 9 & 10 & 11 & 12 & 13 & 14 & 15 & 16 \\ \hline \hline
    $[R(\mc{O}^{\sharp}):\mu_i]$ & 1 & 0 & 3 & 0 & 5 & 0 & 7 & 0 & 8 & 0 & 8 & 0 & 7 & 0 & 5 & 0 & 2 \\ \hline
    $[R(\mc{O}):\mu_i] $ & 1 & 0 & 3 & 0 & 6 & 0 & 9 & 0 & 12 & 0 & 13 & 0 & 12 & 0 & 8 & 0 & 3 \\ \hline
  \end{tabular}
\end{center}}
\normalsize{The discrepancies of the two rows of numbers reflects the non-normality of $\overline{\mc{O}}$.}}\\
\end{example}
We would like to end with a conjecture:
\begin{conjecture}
Let $\mc{O}$ be a classical nilpotent orbit for $G$, and $\mu$ is any irreducible, finite dimensional representation of $G$. Then the multiplicities $[R(\overline{\mc{O}}):\mu]$ can be computed. In particular, if $\mu = \mu_i$ is a fundamental representation, then
$$[R(\overline{\mc{O}}):\mu] = [R(\mc{O}^{\sharp}) : \mu]$$
\end{conjecture}
The proof of the above Conjecture for $\mc{O} = (2a_{2k}, \dots, 2a_1, 2a_0)$ is the content of an on-going work of Barbasch and the author in \cite{BW}.


\begin{thebibliography}{30}

\bibitem{Ac1}
P. Achar,
\emph{Equivariant Coherent Sheaves on the Nilpotent Cone for Complex Reductive Lie Groups},
Ph.D. Thesis, Massachusetts Institute of Technology
2001

\bibitem{Ac2}
P. Achar,
\emph{An order-reversing duality map for conjugacy classes in Lusztig's canonical quotient},
Transform. Groups 8, 107-145
2003

\bibitem{AS}
P. Achar and E. Sommers,
\emph{Local systems on nilpotent orbits and weighted Dynkin diagrams},
Represent. Theory 6, 190-201
2002

\bibitem{AB}
J. Adams and D. Barbasch,
\emph{Reductive Dual Pair Correspondence for Complex Groups},
J. Funct. Anal. \textbf{132}, 1-42,
1995

\bibitem{AHV}
J. Adams, J.-S. Huang, D. Vogan,
\emph{Functions on the model orbit in E8},
Represent. Theory \textbf{2}, 224-263,
1998

\bibitem{AO 2004}
J-P. Anker and B. Orsted,
\emph{Lie theory: Lie algebras and representations},
Birkhauser,
2004

\bibitem{BV 1985}
D. Barbasch and D. Vogan,
\emph{Unipotent Representations of Complex Semisimple Groups},
Ann. of Math. \textbf{121}, No.1, 41-110,
1985

\bibitem{B 1989}
D. Barbasch,
\emph{The Unitary Dual for Complex Classical Lie Groups},
Invent. Math. \textbf{96}, 103-176,
1989

\bibitem{B 2008}
D. Barbasch,
\emph{Regular Functions on Covers of Nilpotent Coadjoint Orbits},
\texttt{http://arxiv.org/abs/0810.0688v1},
2008

\bibitem{B 2015}
D. Barbasch,
\emph{Unipotent representations and the Theta-correspondence},
Conference in honor of Professor Roger Howe,
\texttt{http://math.mit.edu/conferences/howe/slides/Barbasch.pdf},
2015


\bibitem{BW}
D. Barbasch and K. Wong,
\emph{Regular Functions of Symplectic Nilpotent Orbit Closures and their Normality},
in preparation.

\bibitem{Br 2003}
R. Brylinski,
\emph{Dixmier Algebras for Classical Complex Nilpotent Orbits via Kraft-Procesi Models I},
The orbit method in geometry and physics: in honor of A.A. Kirillov, Birkhauser,
2003


\bibitem{Car}
R. Carter,
\emph{Finite Groups of Lie Type},
Wiley \& Sons,
1993

\bibitem{CM}
D. Collingwood and W. McGovern,
\emph{Nilpotent orbits in semisimple Lie algebras},
Van Norstrand Reinhold Mathematics Series,
1993

\bibitem{CO}
T. Chmutova and V. Ostrik,
\emph{Calculating canonical distinguished involutions in the affine Weyl groups},
Exp. Math. \textbf{11}, 99-117,
2002

\bibitem{KV}
M. Kashiwara and M. Vergne,
\emph{On the Segal-Shale-Weil representations and harmonic polynomials},
Invent. Math. \textbf{44}, 1-47,
1978


\bibitem{KP 1979}
H. Kraft and C. Procesi,
\emph{Closures of Conjugacy Classes of Matrices are Normal},
Invent. Math. \textbf{53}, 227-247,
1979

\bibitem{KP 1982}
H. Kraft and C. Procesi,
\emph{On the Geomery of Conjugacy Classes in Classical Groups},
Comment. Math. Helv. \textbf{57}, 539-602,
1982

\bibitem{LZ}
K. Liang and M. Zhang,
\emph{Induced orbit data and induced unitary representations for complex groups},
Sci. China Math. \textbf{57}(7), 1435-1442,
2014


\bibitem{Lu1}
G. Lusztig,
\emph{Characters of reductive groups over a finite field},
Ann. Math. Studies \textbf{107},
1984

\bibitem{Lu2}
G. Lusztig,
\emph{Notes on unipotent classes},
Asian J. Math. \textbf{1}, 194-207,
1997

\bibitem{MG1}
W. McGovern,
\emph{Rings of regular functions on nilpotent orbits and their covers},
Invent. Math. \textbf{97}, 209-217,
1989

\bibitem{MG3}
W. McGovern,
\emph{Completely Prime Maximal Ideals and Quantization},
Mem. Amer. Math. Soc. \textbf{519},
1994

\bibitem{S1}
E. Sommers,
\emph{A generalization of the Bala-Carter theorem},
Int. Math. Res. Not., 539-562,
1998

\bibitem{S2}
E. Sommers,
\emph{Lusztig's canonical quotient and generalized duality},
J. Algebra \textbf{243}, 790-812,
2001

\bibitem{V2}
D. Vogan,
\emph{Associated Varieties and Unipotent Representations},
Harmonic Analysis on Reductive Groups (W. Barker and P. Sally, eds.), Birkhauser, Boston-Basel Berlin,
1991


\bibitem{V1998}
D. Vogan,
\emph{The method of coadjoint orbits for real reductive groups},
IAS/Park City Math. Ser. \textbf{6},
1998

\bibitem{V2015}
D. Vogan,
\emph{Coherent sheaves on nilpotent cones},
\texttt{http://www-math.mit.edu/~dav/dubrovnikHO.pdf},
2015

\bibitem{W2}
K. Wong,
\emph{Dixmier Algebras on Complex Classical Nilpotent Orbits and their Representation Theories},
Ph.D. Thesis, Cornell University,
2013

\bibitem{W 2016}
K. Wong,
\emph{Regular Functions of Spherical Symplectic Nilpotent Orbits and their Quantizations},
Represent. Theory. \textbf{19}, 333-346,
2015

\bibitem{Yang}
L. Yang,
\emph{On the quantization of spherical nilpotent orbits},
Trans. Amer. Math. Soc. \textbf{365}, 6499-6515,
2013

\end{thebibliography}
\end{document}